\documentclass{amsart}
\usepackage{amssymb}
\usepackage{mathrsfs}
\usepackage{graphicx}
\usepackage{comment}

\newcommand{\ve}{\varepsilon}

\def\S{\mathcal S}

\newcommand{\be}{\begin{equation}}
\newcommand{\ee}{\end{equation}}
\newcommand{\ba}{\begin{align}}
\newcommand{\ea}{\end{align}}

\newtheorem{theorem}{Theorem}[section]
\newtheorem{corollary}[theorem]{Corollary}
\newtheorem{lemma}{Lemma}[section]

{\begin{list}{}{%
\settowidth{\labelwidth}{\textsf{{\it #1.}}}%
\setlength{\labelsep}{4mm}%
\setlength{\leftmargin}{\labelwidth}%
\addtolength{\leftmargin}{\labelsep}%
}}%
{\end{list}}

\def\om{\omega}
\def\beq{\begin{equation}}\def\enq{\end{equation}}

{\begin{list}{}{%
\settowidth{\labelwidth}{\textsf{{\it #1.}}}%
\setlength{\labelsep}{2mm}%
\setlength{\leftmargin}{\labelwidth}%
\addtolength{\leftmargin}{\labelsep}%
\addtolength{\leftmargin}{4mm}%
\setlength{\itemsep}{6pt}%
\setlength{\listparindent}{0pt}%
\setlength{\topsep}{3pt}%
}}%

\title{Minimal group determinants for dicyclic groups}

\author[B. Paudel]{Bishnu Paudel}
\address{ Department of Mathematics\\
         Kansas State University\\
         Manhattan, KS 66506, USA}
\email{bpaudel@ksu.edu, pinner@math.ksu.edu}

\author[C. Pinner]{Christopher Pinner}

\date{\today}
\begin{document}

\begin{abstract}
We determine the minimal non-trivial  integer group determinant  for the dicyclic group of order $4n$ when  $n$ is odd.
We also discuss  the set of all  integer group determinants for the dicyclic groups of order $4p$.

\end{abstract}

\maketitle

\section{Introduction}\label{secIntroduction}
For a finite group $G=\{g_1,\ldots ,g_n\}$ of order $n$ we assign a  variable $x_g$ for each element $g\in G$ and define the {\it group determinant}
$\mathscr{D}_G(x_{g_1},\ldots ,x_{g_n})$ to be the determinant of the $n\times n$ matrix whose $ij$th entry is $x_{g_i g_j^{-1}}$.
We are interested here in the values that the group determinant can take when the variables are all integers
$$ \S(G) =\{ \mathscr{D}_G(x_{g_1},\ldots ,x_{g_n})\; : \; x_{g_1},\ldots ,x_{g_n}\in \mathbb Z\}. $$
Notice that $\S(G)$ will be closed under multiplication:
\be \label{mult}  \mathscr{D}_G(a_{g_1},\ldots ,a_{g_n})\mathscr{D}_G(b_{g_1},\ldots ,b_{g_n})=\mathscr{D}_G(c_{g_1},\ldots ,c_{g_n}), \;\;\; c_g=\sum_{uv=g} a_ub_v. \ee

An old problem of Olga Taussky-Todd 
is to determine  $\S (\mathbb Z_n)$, where the group determinants are the $n\times n$ circulant determinants with integer entries. Here and throughout we write $\mathbb Z_n$ for the integers modulo $n$, and $p$ will always denote a prime.

Laquer \cite{Laquer} and Newman \cite{Newman1,Newman2} obtained divisibility conditions on the values of the group determinant
for integer variables for cyclic groups and a complete description  of the values for certain cyclic groups. 
For example,   Laquer \cite{Laquer} and Newman \cite{Newman1} showed that
\be \label{Zp}  \S (\mathbb Z_p)=\{ p^a m,\;\; (m,p)=1,\;\; a=0 \text{ or } a\geq 2\}, \ee
and Laquer \cite{Laquer} that for odd $p$
\be \label{Z2p}  \S (\mathbb Z_{2p})=\{ 2^a p^b m,\;\; (m,2p)=1,\;\; a=0 \text{ or } a\geq 2, \;\; b=0 \text{ or } b\geq 2 \}. \ee
Newman \cite{Newman2} described $\S(\mathbb Z_9)$  with upper and lower set  inclusions for general $\mathbb Z_{p^2}$. 
For the general cyclic group Newman \cite{Newman1} showed that
\be \label{coprime} \{m\in \mathbb Z\; :\; \gcd(m,n)=1\}\subset \S(\mathbb Z_n), \ee
with a divisibility restriction for the values not coprime to the order:
\be \label{abeliandiv}  p^t\parallel n,\; p\mid m\in \S(\mathbb Z_n)\;\; \Rightarrow \;\; p^{t+1}\mid m. \ee
For odd $p$ the values for the Dihedral groups of order $2p$
or $4p$  were obtained in  \cite{dihedral}: 
\begin{align*}
\S( D_{2p}) & =\{ 2^a p^b m\; : \; (m,2p)=1,\; a=0 \text{ or } a\geq 2, \; b=0 \text{ or }  b\geq 3 \}, \\
\S (D_{4p}) & = \{m\equiv 1 \text{ mod } 4 \; : \; p\nmid m \text{ or } p^3\mid m\}  \\
  & \hspace{5ex} \cup \{ 2^a p^b m\; : \; (m,2p)=1,\; a=4 \text{ or } a\geq 6, \; b=0 \text{ or }  b\geq 3 \}, 
\end{align*}
with a counterpart to \eqref{coprime}
\be \label{coprimedihedral} \{m\in \mathbb Z\; :\; \gcd(m,2n)=1\}\subset \S(D_{2n}), \ee
for $n$ odd, but only those  $\gcd(m,2n)=1$ with $m\equiv 1$ mod $4$ when $n$ is even, 
and the divisibility  condition \eqref{abeliandiv}
\be \label{dihedraldiv}  p^t\parallel n,\; p\mid m\in \S( D_{2n})\;\; \Rightarrow \;\; p^{2t+1}\mid m, \ee
for odd $p$, with $2^2,2^4$ or $2^{2t+4}\mid m$ when $p=2$ and $t=0,1$ or $t\geq 2$ respectively.

A complete description for all groups of order at most 14 was found in \cite{smallgps} and for $S_4$ in \cite{S4}.
 For example for the two dicyclic groups of order less than 14:
\be \label{Q8} 
\S(Q_8)  = \{ 8m+1, \; (8m-3)p^2, \text{ and } 2^8 m\; : \; m \in \mathbb Z, \; p\equiv 3 \text{ mod } 4 \} 
\ee
and
\begin{align}\label{Q12} \S(Q_{12}) & = \{2^a 3^b m\;:\; a=0, \; 4 \text{ or } a\geq 6, b=0 \text{ or } b\geq 3, \; \gcd(m,6)=1\}  \\
      &  \;\;\;  \;\; \cup   \{2^5 3^b m\;:\; b=4 \text{ or } b\geq 6, \; \gcd(m,6)=1\}   \nonumber \\
  & \;\;\;\;\; \cup   \{2^5 3^b mp\;:\; b=0, \; 3  \text{ or } 5, \; \gcd(m,6)=1, p\equiv 5 \text{ mod } 12 \}  \nonumber \\
 &\;\;\;\;\;   \cup   \{2^5 3^b m p^2\;:\; b=0, \; 3  \text{ or } 5, \; \gcd(m,6)=1, \; p\equiv 5 \text{ mod }6\}.\nonumber
\end{align}
The complexity encountered even for small groups  \cite{smallgps} makes it clear that obtaining $\S (G)$ is not in general feasible.
Indeed simply finding  the smallest non-trivial integer determinant  
\be \label{deflambda} \lambda (G) := \min \{ |\mathscr{D}_G(x_{g_1},\ldots ,x_{g_n})| \geq 2\; : \; x_{g_i}\in \mathbb Z\} \ee
can be difficult. For a group of order $n$ taking $x_e=0$ and $x_g=1$ for $g\neq e$ always gives determinant  $(-1)^{n-1}(n-1)$, so we have as our trivial bound
\be \label{trivialbd} \lambda(G)\leq |G|-1 \ee
for $|G|\geq 3$, with $\lambda(\{e\})=2$, $\lambda(\mathbb Z_2)=3$.

Kaiblinger  \cite{Norbert}  obtained $\lambda (\mathbb Z_n)$ when $420\nmid n$, with this extended to $2^3\cdot 3\cdot 5\cdot 7\cdot 11\cdot 13\cdot 17\cdot 19\cdot 23  \nmid n$
in \cite{Pigno1}. Values of $\lambda(G)$ for non-cyclic abelian $G$  are considered in \cite{dilum,pgroups,Cid2, Stian,2gp}.
In \cite{dihedral} the value of $\lambda (D_{2n})$ was obtained for any dihedral group of order $2n$ with $2^2\cdot 3^2\cdot 5\cdot 7\cdot 11\cdot 13 \cdot  \ldots \cdot 107\cdot 109\cdot 113 \nmid n$. Our goal here is to determine similar results for $Q_{4n}$, the dicyclic group of order $4n$, when $n$ is odd.

\section{Lind Mahler Measure} 
For a polynomial $F\in \mathbb Z[x,x^{-1}]$  one defines the traditional logarithmic  Mahler measure by
\be \label{MM} m(F)=\int_0^1 \log |F(e^{2\pi  i \theta})| d\theta. \ee
Lind \cite{Lind} regarded this as a measure on the group $\mathbb R/\mathbb Z$ and extended the concept to a compact abelian group with a Haar measure. For example for an $F\in \mathbb Z[x,x^{-1}]$ and cyclic group $\mathbb Z_n$ we can
define a $\mathbb Z_n$-logarithmic measure
$$  m_{\mathbb Z_n}(F)  = \frac{1}{n} \sum_{z^n=1} \log |F(z)|  $$
That is $m_{\mathbb Z_n}(F)= \frac{1}{n} \log |M_{\mathbb Z_n}(F)| $
where 
$$ M_{\mathbb Z_n} (F):= \prod_{j=0}^{n-1} F(w_n^j),\;\;\; w_n:=e^{2\pi i/n}. $$
More generally for a finite abelian group
\be \label{genab} G=\mathbb Z_{n_1} \times \cdots \times \mathbb Z_{n_k} \ee
we can define a logarithmic $G$-measure on $\mathbb Z [ x_1,\ldots ,x_k]$ by
$$ m_G(F)=\frac{1}{|G|} \log |M_G(F)|,\;\;\; M_G(F) =\prod_{j_1=0}^{n_1-1} \cdots \prod_{j_k=0}^{n_k-1} F\left(w_{n_1}^{j_1},\cdots ,w_{n_k}^{j_k}\right). $$
As by observed by Dedekind the group determinant for a finite abelian group can be factored into linear factors using the group
characters $\hat{G}$
\be \label{abelianfactor} \mathscr{D}_G(x_{g_1},\ldots ,x_{g_n})=\prod_{\chi \in \hat{G}} \left(\chi(g_1)x_{g_1}+\cdots + \chi(g_n)x_{g_n}\right), \ee
and can be related directly to a  Lind Mahler  measure for the group, see for example \cite{Cid2}. For example in the cyclic case, see \cite{Norbert2}
$$ \mathscr{D}_{\mathbb Z_n}(a_0,a_1,\ldots ,a_{n-1}) = M_{\mathbb Z_n}(a_0+a_1x+\cdots + a_{n-1}x^{n-1}), $$
and in the general finite abelian case \eqref{genab}
$$ \mathscr{D}_{G}(a_{g_1},\ldots ,a_{g_n}) = M_{G}\left( \sum_{g=(j_1,\ldots ,j_k)\in G} a_g x_1^{j_1}\cdots x_k^{j_k}\right). $$

For a finite non-abelian group the group determinant will not factor into linear factors but can still be factored using the group representations $\hat{G}$
$$ \mathscr{D}_G(x_{g_1},\ldots ,x_{g_n})=\prod_{\rho\in \hat{G}} \det\left( \sum_{g\in G} x_g \rho(g)\right)^{\deg(\rho)} $$
 as discovered by Frobenius, see for example \cite{Formanek,Conrad}. In \cite{dihedral} it was shown that the group determinants 
for the dihedral group of order $2n$,
$$ D_{2n}=\langle x,y\; :\; x^n=1,y^2=1, xy=yx^{-1}\rangle =\{1,x,\cdots ,x^{n-1}, y,yx,\ldots ,yx^{n-1}\}, $$
can be written as a $\mathbb Z_n$-measure
\be  \label{dih} \mathscr{D}_{D_{2n}} (a_0,\ldots ,a_{n-1},b_0,\ldots, b_{n-1})=M_{\mathbb Z_n} \left(f(x)f(x^{-1})-g(x)g(x^{-1})\right), \ee
where 
\be \label{deffg} f(x)=a_0+\cdots +a_{n-1}x^{n-1}, \;\;\; g(x)=b_0+\cdots +b_{n-1}x^{n-1}.\ee
Similarly for the dicyclic group of order $4n$,
$$ Q_{4n}=\langle x,y\; :\; x^{2n}=1,y^2=x^n, xy=yx^{-1}\rangle =\{1,x,\cdots ,x^{2n-1}, y,yx,\ldots ,yx^{2n-1}\}, $$
it was shown in \cite{smallgps} that the group representations give
\be  \label{dic} \mathscr{D}_{Q_{4n}} (a_0,\ldots ,a_{2n-1},b_0,\ldots, b_{2n-1})=M_{\mathbb Z_{2n}} \left(f(x)f(x^{-1})-x^n g(x)g(x^{-1})\right),\ee
where
\be \label{deffgg} f(x)=a_0+\cdots +a_{2n-1}x^{2n-1}, \;\;\; g(x)=b_0+\cdots +b_{2n-1}x^{2n-1}.\ee
Notice that we can conversely use the group determinant  to define a Lind style polynomial measure for non-abelian finite groups. For example 
we can define $D_{2n}$ and $Q_{4n}$ measures on $\mathbb Z[x,y]$
by 
\begin{align}
M_{D_{2n}}(f(x)+yg(x))&=M_{\mathbb Z_n} \left(f(x)f(x^{-1})-g(x)g(x^{-1})\right), \nonumber \\
\label{formulaQ} M_{Q_{4n}}(f(x)+yg(x))&=M_{\mathbb Z_{2n}} \left(f(x)f(x^{-1})-x^n g(x)g(x^{-1})\right), 
\end{align}
  although here the polynomial ring is no longer commutative, the monomials having to satisfy the group relations $y^2=1$, $xy=yx^{-1}$ etc., the relations allowing us to reduce any  $F(x,y)$ to the form $f(x)+yg(x)$, with $f$ and $g$ of the form 
\eqref{deffg} or \eqref{deffgg} if we want, and to multiply and reduce two polynomials.

The classical Lehmer problem \cite{Lehmer} is to determine $\inf\{  m(F)>0\: :\: F\in\mathbb Z [x]\}$. Given the correspondence between the Lind measures and group determinants in the abelian case  we can  regard determining 
$\lambda (G)$ for a finite group as the Lind-Lehmer problem for that group.
An alternative way of extending the Mahler measure to groups can be found in \cite{Lalin}.

\section{Minimal determinants for odd $n$}

For the dicyclic groups $G=Q_{4n}$ we have some extra properties when $n$ is odd. For example, since
 \be \label{minus} M_G(f(x)+yg(x))=(-1)^n M_G(g(x)+yf(x)), \ee
if $n$ is odd we have $-m\in \S (G)$ whenever $m \in \S(G)$. This is certainly not true when $n$ is even
as we saw  for $Q_8$. When $n$ is odd we also have 
$$M_G\left( 1+(x^n+1)(x+\cdots +x^{(n-1)/2})+y(x^n+1)(x+\cdots +x^{(n-1)/2})  \right)=   2n-1, $$
always improving on the trivial bound \eqref{trivialbd}, and 
\be \label{16} M_G(x^2+1)=16, \ee
giving us an absolute bound $\lambda(Q_{4n})\leq 16$ for $n$ odd. In the next section
we will see that an analog to  \eqref{coprime} and \eqref{coprimedihedral} holds for $n$ odd:
\be \label{coprimedicyclic} \{m\in \mathbb Z\; :\; \gcd(m,2n)=1\}\subset \S(Q_{4n}), \ee
and, corresponding to the divisibility conditions \eqref{abeliandiv} and \eqref{dihedraldiv},
\be \label{dicyclicdivproperty} 2\mid m\in \S(Q_{4n})\; \Rightarrow 16\mid m,\;\;\; p^t\parallel n,\; p\mid m\in \S(Q_{4n})\; \Rightarrow\; p^{2t+1}\mid m. \ee

Properties \eqref{16},\eqref{coprimedicyclic},\eqref{dicyclicdivproperty}, are  enough for us to completely  determine $\lambda (Q_{4n})$:

\begin{theorem}\label{odd}  If $n$ is odd then
$$ \lambda (Q_{4n}) =\min\{ 16, p_0\} $$
where $p_0$ is the smallest prime not dividing $2n$.
That is,
$$ \lambda(Q_{4n}) = \begin{cases} 3 & \text{ if $3\nmid n$, } \\
 5  & \text{ if $3\mid n,$ $5\nmid n$, } \\
 7 & \text{ if $3\cdot 5 \mid n,$ $7\nmid n$, } \\
 11 & \text{ if $3\cdot 5\cdot 7\mid n,$ $11 \nmid n$, } \\
 13 & \text{ if $3\cdot 5\cdot 7\cdot 11\mid n,$ $13 \nmid n$, } \\
 16 & \text{ if $3\cdot 5\cdot 7 \cdot 11\cdot 13 \mid n$. } \end{cases}
$$
\end{theorem}

A complete description of the determinants for $D_{2p}$ and $D_{4p}$ was given in \cite{dihedral}. As we saw for $Q_{12}$ in \eqref{Q12} 
the determinants  for  $Q_{4p}$ must depend more subtly on $p$, or at least those determinants $M$ with $2^5\parallel M.$  We can 
be precise about the other values.
 
\begin{theorem} \label{Q4p}
Suppose that $p$ is an odd prime. 
The determinants for $Q_{4p}$ will take the form $2^kp^{\ell}m$, $\gcd(m,2p)=1$, with
$k=0$ or $k\geq 4$ and $\ell=0$ or $\ell\geq 3$.

We can achieve all such values with $k=0$, $k=4$ or $k\geq 6$, and  all  with $k=5$ and $\ell=4$ or $\ell\geq 6$.

This just leaves  $2^5 m$, $2^5 p^3m$, $2^5 p^5m$, $\gcd(m,2p)=1$.  Not all $m$ are possible.
 
\noindent
The smallest determinant of the form  $2^5|m|$, $\gcd(m,2p)=1$ has $m=\frac{1}{2}(p^2+1)$.

\noindent
If $p\equiv 3$ mod 4  the smallest $2^5 p^3|m|$, $2^5  p^5|m|$, $\gcd(m,2p)=1$  have   $m= \frac{1}{2}(p^2+1)$.

\noindent
If $p\equiv 1$ mod 4 then all the multiples of $2^5p^5$ are determinants. For $p=5$  all multiples of  $2^5p^3$  are determinants. 
\end{theorem}
For the $p\equiv 1$ mod 4 with $p>5$ it remains unclear whether we achieve any $2^5p^3m$, $\gcd(m,2p)=1$, with $|m|$ 
smaller than $m=\frac{1}{2}(p^2+1)$.

\section{The case of even $n$}\label{evenn}

When $G=Q_{4n}$ with $n$ even it is not at all obvious which values coprime to $2n$ are determinants;  \eqref{coprimedicyclic} is far from true, the odd determinants must be  $1$ mod $4$ with only some of those obtainable.
 The observation that when $g=0$ we have 
\be \label{y=0} M_{Q_{4n}}(f(x))=M_{\mathbb Z_{2n}}(f(x))^2, \ee
does give us 
$$\{m^2 \; : \; \gcd(m,2n)=1\}\subset \{m^2 \; : \; m\in \S(\mathbb Z_{2n})\}\subset \S (Q_{4n}), $$
where, writing $\Phi_{\ell}(x)$ for the $\ell$th cyclotomic polynomial, 
$$ \gcd(m,2n)=1 \; \Rightarrow \; M_G\left( \prod_{p^{\alpha}\parallel m} \phi_p(x)^{\alpha}\right)=m^2. $$
As a counterpart to \eqref{16} we do have 
\be \label{crude2}  2^{t} \parallel n \; \Rightarrow \; M_{G}(x^{2^{t+1}}+1)=2^{2^{t+2}}. \ee
In particular we always have
\be \label{minp} \lambda(Q_{4n})\leq \min\left\{ 2^{2^{t+2}}, \;p_0^2 \right\},\ee
where $p_0$ is the smallest prime not dividing $2n$. 
With our divisibility conditions and Lemma \ref{neven} we can certainly come up with cases of equality in \eqref{minp},
though not always, for example $\lambda(Q_8)=7$.

In a future paper we hope to consider the case of $Q_{4n}$, $2\parallel n$.
The general case of even $n$ seems far out of reach; for $t=1$ we know that \eqref{crude2} does give the smallest even determinant, but for $t\geq 2$ this is not at all clear, indeed the counterpart for cyclic groups remains unresolved.

\section{Divisibility restrictions and values achieved}

We work with the dicyclic measures of polynomials $F=f(x)+yg(x)$,
$$M_G(F)=M_{\mathbb Z_{2n}}\left(f(x)f(x^{-1}) -x^n g(x)g(x^{-1})\right), \;\;\; f,g \in \mathbb Z [x],$$
where if the degree of $f$ or $g$ exceeds $2n-1$ we can still recover a group determinant
by reducing the polynomial mod $x^{2n}-1$ to the form \eqref{deffgg}.

For the dicyclic determinants  we obtain  a divisibility Lemma very much like that obtained for the cyclic groups 
and the dihedral groups \cite[Lemma 4.4]{dihedral}. We begin by observing that the cyclic results \cite[Theorem 2]{Newman1} and  \cite[Theorem 5.8]{Norbert2} are in fact best possible:

\begin{lemma}\label{cycliclemma} 
Suppose that $p^{\alpha}\parallel n$ then
\be \label{cyclicdiv} p\mid M_{\mathbb Z_n}(F(x)) \Rightarrow p^{\alpha+1}\mid M_{\mathbb Z_n}(F(x)) . \ee
Since 
$$ M_{\mathbb Z_n}\left( x-1 + \frac{x^n-1}{x-1}\right)=n^2, $$
and for odd $p$
$$  p^{\alpha+1} \parallel M_{\mathbb Z_n}\left( p+(x-1)\right), $$
this is sharp for $\alpha \geq 1$ when $p$ is odd and for  $\alpha =1$ when $p=2$. 

For $p=2$ and $\alpha \geq 2$ we have
\be \label{cyclicdiv} 2\mid M_{\mathbb Z_n}(F(x)) \Rightarrow 2^{\alpha+2}\mid M_{\mathbb Z_n}(F(x)) . \ee
Since
\be \label{2example}  2^{\alpha+2}\parallel M_{\mathbb Z_n}\left( 4+ (x-1)\right), \ee
this exponent is again sharp.
\end{lemma}
Although the exponent is sharp we do not necessarily get that prime power itself (let alone all multiples); for example Newman \cite{Newman2} showed that $p^3$ is in $\S(\mathbb Z_{p^2})$ for $p=3$ but not for any  $p\geq 5$.

For the dicyclic groups we have:

\begin{lemma} \label{dicyclicdiv}
Suppose that $G=Q_{4n}$.

(i) For odd  $p,$ if $p^{\alpha}\parallel n$ and $p\mid M_G(F)$ then $p^{2\alpha+1}\mid M_G(F)$. 

This is best possible, for example
$$ p^{2\alpha +1} \parallel M_G\left( 1-(1+x^n)(1-x)+ y\left(\frac{p-1}{2}\right) (1+x^n)\right). $$

(ii) Suppose that $2^{\alpha}|| n$ and $2\mid M_{G}(F)$.

(a) If $\alpha=0$ then $2^4\mid M_G(F)$.

(b) If $\alpha \geq 1$ then  $2^{2\alpha+6}\mid M_G(F)$.

Since 
$$M_{G}\left(x^{2^{\alpha+1}}+1 + m\frac{x^{2n}-1}{x-1} + y m\frac{x^{2n}-1}{x-1}\right)=2^{2^{\alpha+2}}(1+2mn), $$
and
\be \label{upperbound} 2^{2\alpha+6}\parallel  M_{Q_{4n}} \left( 4 + (x-1)\right)   \ee
the exponents  in (a) and (b) are optimal.

\end{lemma}
For $Q_{4p}$, $p$ odd,  we  achieve all odd  multiples of $p^3$, and for $Q_8$ all multiples of $2^8$, but in general  it is not clear whether we can achieve the prime power $p^{1+2\alpha}$ or $2^{2\alpha+6}$ itself. 
Property \eqref{upperbound} is just \eqref{2example} and  \eqref{y=0}.

When $n$ is odd the next lemma, the counterpart to \cite[Lemma 4.2]{dihedral},  shows that we can achieve any integer coprime to $2n$. By
\eqref{mult} and \eqref{minus} it is enough to achieve $p$ or $-p$ for any $p\nmid n$.

\begin{lemma}\label{achieve} Suppose that $n$ is odd  and $p\nmid n$ is an odd prime,  where $p\equiv \delta$ mod $4$ with $\delta=\pm 1$.
Set  $t=(p-\delta)/4$, and
$$ f =\delta + (x^n+1)H(x),\;\;\; g=(x^n+1)H(x), $$
with
$$ H(x)= \left( \frac{ x^{m} +1}{x+1}\right)\left( x^{a_1}+\cdots + x^{a_t}\right), $$
where $m$ is odd with $pm\equiv 1$ mod $n$, and $pa_1,\ldots , pa_t\equiv 1,3,\ldots , (p-3)/2$  mod $n$ if $\delta =1,$
and $0,2,\ldots , (p-3)/2$ mod $n$  if $\delta =-1$, then
$$ M_{Q_{4n}} (f(x)+yg(x)) =\delta p. $$

\end{lemma}

When $2\mid n$  we have additional restrictions on the odd determinants, showing that we can no longer achieve all integers coprime to $2n$ :

\begin{lemma}\label{neven}
Suppose that $G=Q_{4n}$ with $2\mid n$.

 If $2\nmid M_G(F)$ then  $M_G(F)\equiv 1$  or $-3$ mod $8$.

If  $2\parallel n$ and $M_G(F)\equiv -3$ mod 8 then
$M_G(F)=(8m-3)k^2$ for some integer $m$ and positive integer  $k\equiv 3$ mod $4$.
Further we can assume that $\gcd(k,n)=1$ or $M_G(F)=(8m-3)p^4$ with $p\mid n$. In either case if $q\mid \gcd(n,(8m-3))$  is prime then $q^2\mid (8m-3)$. 

If $2\parallel n$ and $M_G(F)\equiv -3$ mod $8$ is of the form $\pm q^{\beta}$ with $q^{\alpha}\parallel n$, $\alpha \geq 1,$ 
then $\beta \geq 4\alpha +3$.
\end{lemma}

\section{Proofs}
We shall need to know the resultant of two cyclotomic polynomials, see \cite{Apostol} or \cite{ELehmer}; if $m>n$ then
$$ |\text{Res}(\Phi_n,\Phi_m)|=\begin{cases}p^{\phi(n)} & \text{ if $m=np^{t}$, }\\
1 & \text{ else.} \end{cases} $$

It will be useful to  split the product over the $2n$th roots of unity in \eqref{formulaQ} into the primitive $d$th roots of unity with $d\mid 2n$:
$$ M_G(F) = \prod_{d\mid 2n} M_d, $$
where
$$ M_d := \prod_{\stackrel{j=1}{(j,d)=1}}^{d} f(w_d^j)f(w_d^{-j})-w_d^{nj} g(w_d^j)g(w_d^{-j}),\;\;\; w_d:=e^{2\pi i/d}. $$
Since we run through complete sets of conjugates the $M_d$ are integers. Moreover, since $f(x)f(x^{-1})-x^ng(x)g(x^{-1})$
is fixed by $x\mapsto x^{-1}$, $x^{2n}=1$, when $d\neq 1,2$ we run through a complete set of conjugates twice and $M_d$
will actually be the square of an integer for $d\geq 3$.

\begin{proof}[Proof of Lemma \ref{cycliclemma}] Suppose that $G=\mathbb Z_n$ and write
$$ M_G(f)=\prod_{d\mid n} U_d(f),\;\; U_d(f)=\text{Res}(\Phi_d,f)\in \mathbb Z. $$
Suppose $p\mid M_G(f)$ then $p\mid U_{mp^j}(f)$ some $p\nmid m$, $0\leq j\leq \alpha$, and since $(1-w_{p^j})\mid p$ we 
have 
$$U_{mp^j}(f)=\prod_{\stackrel{r=1}{\gcd(r,m)=1}}^m \prod_{\stackrel{s=1}{\gcd(s,p)=1}}^{p^j} f(w_{m}^rw_{p^j}^s) \equiv U_m(f)^{\phi(p^j)} \text{ mod } p$$ and $p\mid U_{mp^j}(f)$ all $j=0,\ldots ,\alpha,$ and $p^{\alpha+1}\mid M_G(f)$.

Observe that $F(x)=\prod_{\stackrel{r=1}{(r,m)=1}}^m f(w_m^rx)$ is in $\mathbb Z[x]$ (since, for example,  its coefficients are fixed by  the automorphisms of $\mathbb Q(w_m)$). Hence when $p=2$ and $\alpha\geq 2$ we can write
$$U_m(f)U_{2m}(f)U_{4m}(f)=U_1(F)U_2(F)U_4(F)=M_{\mathbb Z_{4}}(F). $$
From \cite{Norbert2}  we have 
$$\S(\mathbb Z_4)=\{ 2^a c \; :\; \gcd(c,2)=1, \; a=0 \text{ or } a\geq 4\}. $$
 Hence we have
$2^4\mid U_m(f)U_{2m}(f)U_{4m}(f)$ and $2\mid U_{mp^j}$ any $j=3,\ldots ,\alpha$, and $2^{\alpha+2}\mid M_G(F).$

For the examples observe that $U_{d}(p+(x-1))\equiv U_d(x-1) = \Phi_d(1)\not \equiv 0$ mod $p$ unless $d$ is a power of $p$,
while $U_1(p+x-1)=p$,  for the $d=p^j$, $j=1,\ldots ,\alpha$ and $p\geq 3$ and $x$ a primitive $p^j$th root of unity we can write $p+(x-1)=(x-1)v$, $v\equiv 1$ mod $(1-w_{p^j})$ and $U_{p^j}(p+(x-1))=U_{p^j}(x-1) (1+tp)=p(1+tp)$ and $p\parallel U_{p^j}(p+(x-1))$ and $p^{\alpha+1}\parallel M_G(p+(x-1))$. The case $p=2$ and $M_G(4+(x-1))$ is similar,  except that $2^2\parallel U_1(4+(x-1))$. 

\end{proof}

\begin{proof}[Proof of Lemma \ref{dicyclicdiv}]
Observe that  if $d=mp^j$ with $\gcd(m,p)=1$ then the primitive $dth$  roots of unity can be written in the form $w_m^r w_{p^j}^s,$ $r=1,...,m$, $\gcd(r,m)=1$ and $s=1,...,p^j$, $\gcd(p,s)=1$. Notice that  $w_m^r w_{p^j}^s\equiv w_m^r$ mod $(1-w_{p^j})$ where $|1-w_{p^j}|_p=p^{-1/\phi(p^j)}$. Hence we have a mod $(1-w_{p^j})$ congruence relating $M_{mp^j}$ and $M_m$ and, since we are dealing 
with integers, actually a  mod $p$ congruence:
\be \label{cong}  M_{mp^j} \equiv M_m^{\phi(p^j)} \text{ mod } p. \ee

Suppose that $p^{\alpha} \parallel n$  and $p\mid M_G(F)$. Then $p\mid M_{mp^j}$ for some $mp^j\mid 2n$, $\gcd(m,p)=1$
and $0\leq j\leq \alpha$ for $p\geq 3$ and $0\leq j \leq \alpha +1$ for $p=2$. 
By \eqref{cong} we get that $p\mid M_{mp^j}$ for all these $j$ and hence $p^2\mid M_{mp^j}$ for all the $j$ if $m>2$
and for $j\geq 1$ if $m=1$ or $2$ and $p\geq 3$ and $j\geq 2$ if $m=1$ and $p=2$.

Hence for $p$ odd and $\alpha \geq 1$ we get $p\mid M_m, p^2\mid M_{mp},\ldots ,M_{mp^{\alpha}}$ and $p^{1+2\alpha}\mid M_G$, improving to $p^{2+2\alpha}\mid M_G$ except when $m=1$ or $2$.

Suppose that $p=2$ and write $n=2^{\alpha}N$. 

Suppose first that $\alpha=0$. If $m>1$ then $2^2\mid M_{m},M_{2m}$ and $2^4\mid M_G(F)$. If $m=1$ then
$$ M_1=f(1)^2-g(1)^2, \;\; M_2=f(-1)^2+g(-1)^2 $$
where $f(1),g(1),f(-1)$ and $g(-1)$ must have the same parity. If both are odd then $2^3\mid M_1$ and $2\parallel M_2$,
while if both are even $2^2\mid M_1,M_2$. Hence in either case $2^4\mid M_G(F)$.

Suppose that $\alpha \geq 1$. We write 
$$ M_G(F) = AB,\;\; A =\prod_{d\mid n} M_{d},\;\; B =\prod_{d\mid N} M_{d2^{\alpha+1}} $$
where, since $M_{m}$ is in $A$ and $M_{m2^{\alpha+1}}$ is in $B$ both are even, with $2^{2\beta}\parallel  B$ since the $M_{d2^{\alpha+1}}$ are squares.
Now
$$ A=M_{\mathbb Z_{n}} \left(f(x)f(x^{-1}) - g(x)g(x^{-1})\right)  = M_{D_{2n}}(F), $$
and it was shown in  \cite[Lemma 4.4]{dihedral} that even $M_{D_{2n}}(F)$ had $2^4\parallel A$ or $2^6\mid A$ if $\alpha=1$
and $2^{2\alpha+4}\mid A$ if $\alpha \geq 2$, giving us $2^6\parallel AB$ or $2^8\mid AB$ when $\alpha=1$ and (b) 
when $\alpha \geq 2$.
It remains to show that we do not have $2^6\parallel M_G(F)$ when $\alpha=1$. 

If $m=1$ then $2^6\parallel M_1M_2M_4=M_{Q_8}(F)$, but from \eqref{Q8} this can not occur. So suppose that for some odd $m\geq 3$ we have $2^2\parallel M_{m},M_{2m},M_{4m}$.
Write:
\begin{align*}  H(x)=\prod_{\stackrel{j=1}{\gcd(j,m)=1}}^{(m-1)/2}  &  \left( f(w_{m}^j x)f(w_m^{-j}x^{-1}) 
-x^n g(w_m^j x)g(w_m^{-j}x^{-1})\right)   \\
 & \times \left( f(w_{m}^{-j} x)f(w_m^{j}x^{-1})-x^{-n} g(w_m^{-j }x)g(w_m^{j}x^{-1})\right).   \end{align*}
and observe that $M_m=H(1),M_{2m}=H(-1), M_{4m}=H(i)^2. $
Observe that $H(x^{-1})=H(x)$, so $H(x)$ is a sum of terms $a_i(x^i+x^{-i})$ and hence 
$$ H(x)=A_0+\sum_{j=1}^N A_j(x+x^{-1})^j, \;\;\; A_j\in \mathbb Z. $$
So
$$ M_{m} \equiv A_0+2A_1+4A_2 \text{ mod }8,\;\;\;  M_{2m} \equiv A_0-2A_1+4A_2 \text{ mod }8,$$
and if $2^2\parallel M_m,M_{2m}$
$$ 2A_0 \equiv M_{m} +M_{2m } \equiv 0 \text{ mod }8. $$
Hence $A_0\equiv 0$ mod $4$ and $4^2\mid M_{4m}$ and $2^8\mid M_G(F)$.

Suppose that $p$ is odd and $F=f+yg$ with  
$$f(x)=1-(1+x^n)(1-x),\;\;\; g(x)=\left(\frac{p-1}{2}\right) (1+x^n), $$
then for $x^n=-1$ or   $x^n=1$ we have
$$ f(x)f(x^{-1})-x^n g(x)g(x^{-1})=1   \text{ or } -2x^{-1}(x-1)^2+2p-p^2,$$
and 
$$ M_G(F)=\prod_{d\mid n} M_d,\;\; M_d=\text{Res}( -2x^{-1}(x-1)^2+2p-p^2, \Phi_d). $$
Now $M_d \equiv 2^{\phi(d)} \text{Res}(\Phi_1,\Phi_d)^2  \not\equiv 0 \text{ mod }p$ unless $d=1,p,\ldots ,p^{\alpha}$.

Plainly $p\parallel M_1=p(2-p)$. 
Since 
$$p=\prod_{\stackrel{u=1}{\gcd(u,p)=1}}^{p^j} (1-w_{p^j}^u) =(1-w_{p^j})^{\phi (p^j)}A(w_{p^j}), $$ 
for $\phi(p^j)>2$ and $x=w_{p^j}^u$ we have
$$  -2x^{-1}(x-1)^2+2p-p^2=(x-1)^2\ell(x),\;\; \ell(x)\equiv -2 \text{ mod } 1-\omega_{p^j}, $$
and
$$ M_{p^j}=\text{Res}(1-x,\Phi_{p^j})^2 L=p^2L, $$
where
$$ L=\prod_{\stackrel{u=1}{\gcd(u,p)=1}}^{p^j} \ell(\omega_{p^j}^u)\equiv (-2)^{\phi(p^j)} \equiv 1 \text{ mod } 1-\omega_{p^j}. $$
Since it is an integer, $L\equiv 1$ mod $p$.
When $\phi(p^j)=2$, that is $p=3$, $j=1$, one has $M_3=3^2$.
Hence $p^2\parallel M_{p^j}$, $j=1,\ldots ,p^{\alpha}$ and $p^{1+2\alpha} \parallel M_G.$

\end{proof}

\begin{proof}[Proof of Lemma \ref{achieve}]
We set $H(x)= \left( \frac{ x^{m} +1}{x+1}\right)\left( x^{a_1}+\cdots + x^{a_t}\right)$ and 
$$B(x)=f(x)f(x^{-1})-x^ng(x)g(x^{-1}). $$
For the values with $x^n=-1$ we plainly have $B(x)=\delta^2=1$ and when $x^n=1$
$$ B(x)= (\delta +2H(x))(\delta +2H(x^{-1}) -  4H(x)H(x^{-1}) = 1+2\delta (H(x)+H(x^{-1})). $$
Notice that if $x=1$ then  $B(x)=1+4\delta H(1)= 1+4\delta t=\delta p$, and since $2\nmid n$
$$\prod_{x^n=1,x\neq 1} (x+1) = \prod_{d\mid n,d\neq 1} \text{Res}(\Phi_d(x),\Phi_{2}(x))=1, $$
so we have 
$$ M_G(f+yg)=M_{\mathbb Z_{2n}}(B(x)) = M_{\mathbb Z_n}(B(x))=( \delta p) M',$$
where
\begin{align*}
M' & = \prod_{x^n=1,x\neq 1} (x+1)  (1+2\delta (H(x)+H(x^{-1}))\\
 & =  \prod_{x^n=1,x\neq 1} \left( x+1 + 2\delta (x^m+1)(x^{a_1}+\cdots x^{a_t}) + 2\delta (x^{-m}+1) (x^{1-a_1}+ \cdots + x^{1-a_t}) \right). 
\end{align*}
As $p\nmid n$  the values of $x^p$ run through the $n$th  roots of unity as $x$ does and
$$
M' = \prod_{x^n=1,x\neq 1} \left( x^p+1 + 2\delta  (x^{mp}+1)(x^{p a_1}+\cdots x^{p a_t}) + 2\delta (x^{-pm}+1) (x^{p-p a_1}+ \cdots + x^{p - pa_t}) \right). $$
Taking  $mp=1$ mod $n$
$$ M' =\prod_{x^n=1,x\neq 1} (x+1) \left( \frac{x^p+1}{x+1} + 2\delta \left( x^{pa_1}+ x^{p-1-pa_1}+ \cdots +  x^{pa_t}+ x^{p-1-pa_t}\right)\right) $$
When $\delta =1$ taking $pa_1,\ldots ,pa_t \equiv 1,3,\ldots ,(p-3)/2 \text{ mod }n$ gives
\begin{align*}  \frac{x^p+1}{x+1} +  & 2\delta \left( x^{pa_1}+ x^{p-1-pa_1}+ \cdots +  x^{pa_t}+ x^{p-1-pa_t}\right)  \\ & =1+x+\cdots + x^{p-1} 
  =\Phi_p(x),\end{align*}
and when $\delta =-1$ taking $pa_1,\ldots ,pa_t \equiv 0,2,\ldots ,(p-3)/2 \text{ mod }n$ gives $-\Phi_p(x)$.
Since $p\nmid n$ we have
$$\prod_{x^n=1,x\neq 1}\Phi_{p}(x)= \prod_{d\mid n,d\neq 1}\text{Res}\left(\Phi_d(x),\Phi_{p}(x)\right)=1$$ 
and $M'=1$.
\end{proof}

\begin{proof}[Proof of Lemma \ref{neven}] Suppose that $M_G(F)$ is odd. We write $M_G(F)=\prod_{d\mid 2n}M_d$. 
Then, since the $M_d$ are odd squares for $d>2$, and so $1$ mod $8$,  we have $M_G(F)\equiv M_1M_2$ mod $8$ where
$M_1=f(1)^2-g(1)^2$, $M_2=f(-1)^2-g(-1)^2$.  Since $M_1$ is odd the $f(1)$, $g(1)$ have opposite parity. Suppose that
$f(1)$ is odd and $g(1)$ even (else switch $f$ and $g$). If $2\parallel  g(1), g(-1)$ then $M_1,M_2\equiv 1-4=-3$ mod $8$ and if $4\mid g(1), g(-1)$ then $M_1,M_2\equiv 1$ mod $8$,  and in both cases $M_G(F)\equiv 1$ mod $8$. If $4\mid g(1)$ and $2\parallel g(-1)$ (or vice versa) then $M_1M_2\equiv -3$ mod $8$ and  $M_G(F)\equiv -3$ mod $8$.

Suppose that $2\parallel n$ and  $M_G(F)\equiv -3$ mod $8$ then we can write
$M_G(F)=(8m-3)M_4$ where $M_4=k^2$, $k=|f(i)|^2+|g(i)|^2,$
where from above we can assume that $f(1)$ is odd, $4\mid g(1)$, $2\parallel g(-1)$ (or vice versa).
Now $|f(i)|^2\equiv f(1)^2$ mod $2$ is odd and of the form $a^2+b^2$ so must be $1$ mod $4$. Separating the monomials into the exponents mod $4$ we have
$g(1)=a_0+a_1+a_2+a_3$, $g(-1)=a_0-a_1+a_2-a_3$, $|g(i)|^2=(a_0-a_2)^2+(a_1-a_3)^2$.
Since $4\mid g(1)$, $2\parallel g(-1)$ (or vice versa) we have $a_0+a_2=\frac{1}{2}(g(1)+g(-1))$, $a_1+a_3=\frac{1}{2}(g(1)-g(-1))$ both odd. So $a_0-a_2$ and $a_1-a_3$ are both odd and $|g(i)|^2\equiv 2$ mod $8$ and $k\equiv 3$ mod $4$.

Now if $p\mid k$ and $p\mid n$ then $p^2\mid M_4, M_{4p}$ and  so $p^4\mid M_4M_{4p}$. In either case if $q\mid (8m-3)$ and $n$, then
either $q\mid M_1$ or $M_2$ and $q^3\mid M_1M_{q}$ or $M_2 M_{2q}$ or the extra $q$ came from a square  $M_d$ with $d>2$ so we must have at least two extra $q$'s.

Suppose $M_G(F)=\pm q^{\beta}\equiv -3$ mod 8, with $q^{\alpha}\parallel n$, $\alpha\geq 1$. Since $\beta$ is odd we must have $q\mid M_1M_2$ and  so $q^{1+2\alpha} \mid M_1 M_q \cdots M_{q^{\alpha}}$  or $M_2 M_{2q} \cdots M_{2q^{\alpha}}$ in addition to the $q^{2+2\alpha} \mid M_4 M_{4q} \cdots M_{4q^{\alpha}}$.
\end{proof}

\begin{proof}[Proof of Theorem \ref{odd}] Suppose that $n$ is odd. From Lemma \ref{dicyclicdiv} we can achieve 16
and from Lemma \ref{achieve} achieve the smallest odd prime $p\nmid n$. The minimum of these is the value claimed for $\lambda(G)$.
By Lemma \ref{dicyclicdiv} an even determinant must be a multiple of 16 and a  value containing a prime $p\mid n$ must be divisible by  $p^3$  (and so be at least  $27$).  Hence we can't beat 16 or the smallest odd prime $p\nmid n$.
\end{proof}

\begin{proof}[Proof of Theorem \ref{Q4p}] From Lemmas \ref{dicyclicdiv} we know that the determinants  must be of the form $2^kp^{\ell} m$, $\gcd(m,2p)=1$, with $k=0$ or $k\geq 4$ and $\ell=0$ or $\ell\geq 3$.
By Lemma \ref{achieve} we can obtain 
all the $m$ with $\gcd(m,2p)=1$, so  by multiplication it  will be enough to achieve the appropriate $2^kp^{\ell}$.

We get  the even  powers $2^k$, $k\geq 4$,  from $g(x)=0$ and
$$
f(x) =x^2+1  \Rightarrow M=2^4,\;\;\; 
f(x)  =x^2+1+(x^p+1)x  \Rightarrow M=2^6,
$$
and the odd powers $k\geq 7$  from $g(x)=(x^p+1)$ and
\begin{align*}
f(x) & =x^4 +1 + (x^p+1)(x^2+x)\;\;\; 
  \Rightarrow  \;\;\; M=2^7,\\
f(x) & =(x^4+1)(x^2+1)+x^2(x^p+1)
\;\;\; \Rightarrow \;\;\;  M=2^9,
\end{align*}
where to see that the $p$th roots give 1 it may be useful to note that
\begin{align*}
\left( x^4+1+2(x^2+x)\right)\left( x^{-4}+1+2(x^{-2}+x^{-1})\right) -4 & =x^{-4}(x+1)^2(x^2+1)^3,\\
\left( (x^4+1)(x^2+1)+2x^2\right)\left( (x^{-4}+1)(x^{-2}+1)+2x^{-2}\right) -4 & =x^{-6}(x^4+1)(x^2+1)^4.
\end{align*}

For the powers of $p$ we write $p=4b+\delta$, $\delta=\pm 1$, $a=2b+\delta$. Then
$$ f(x)=\frac{(x^a-1)}{(x-1)} +mh(x),\; g(x) =(x^p+1) \frac{   (x^b-1)}{(x-1)} +mh(x)\;\;\; \Rightarrow M=\delta p^3 (1+4m), $$
where as usual $h(x)=(x^{2p}-1)/(x-1)$,
giving $\pm p^{\ell}$ for  all the  $\ell\geq 3$ with  suitable choices of $m$. To see that the $p$th roots give $p^2$  observe that $p-a=2b$ and
$$ (x^{p-a}-1)(x^{-(p-a)}-1) -4(x^b-1)(x^{-b}-1) =- (x^b-1)^2(x^{-b}-1)^2.$$

We  get the $2^5 p^{\ell}$ with $\ell=4$ or $\ell \geq 6$ using $p^3$ and  $M=-2^5 p^{2t+4}$, $t\geq 0,$ from
$$f(x)=1-x^2 +2\Phi_p(x^2)^{t+1}-p^t h(x),\;
g(x)=(x^p+1) +2\Phi_p(x^2)^{t+1}-p^t h(x). $$

Finally, suppose that we have a determinant $M=2^5m$, or when $p=3$ mod 4  an $M=2^5p^3m$ or $2^5p^5m$, with $\gcd(m,2p)=1$ 
and  $1\leq |m|<\frac{1}{2}(p^2+1)$.

We write $M=M_1M_2M_pM_{2p}$ where
\begin{align*} M_1=f(1)^2-g(1)^2, \;\;&  M_2=f(-1)^2+g(-1)^2,\;\; \\
M_p=\prod_{j=1}^{p-1} |f(\omega^j)|^2-|g(\omega^j)|^2,\;\; &
 M_{2p}=\prod_{j=1}^{p-1} |f(-\omega^j)|^2+|g(-\omega^j)|^2,\;\; \omega:=e^{2\pi i/p}. 
\end{align*}
Since $M_{p}$, $M_{2p}$ are squares we must have $M_1M_2$ even. Thus $f(1)$, $g(1)$ have the same parity and  $2^4\mid M_1M_2$ and $M_p$, $M_{2p}$ are odd. Likewise when $p\equiv 3$ mod 4 we 
 know that a sum of two squares must be divisible by an even power of  $p$ and so the multiples of $p^3$ and $p^5$ 
must have $p\mid M_1$, $p^2\mid  M_{p}$ and $p\nmid M_2M_{2p}$.
Now $M_{2p}\equiv M_{2}^{p-1} \equiv 1$ mod $p$ and so $M_{2p}=1$, else $m$ is divisible by the square of an odd integer
$\equiv \pm 1$ mod $p$ and $|m|\geq (2p-1)^2$. But $M_{2p} \geq \prod_{j=1}^{p-1} |f(-\omega^j)|^2+ \prod_{j=1}^{p-1} |g(-\omega^j)|^2$, so one of these integers must be zero, say $g(-\omega)=0$. Hence $g(x)=\Phi_p(-x)g_1(x)$.
This gives $g(-1)=pg_1(-1)$ and hence $g_1(-1)=0$, otherwise $m$ has a factor of size at least $\frac{1}{2}(p^2+1)$.
Hence $M_2=f(-1)^2$ is divisible by an even power $2^{2t}$, $t\geq 1$. But  $g(1)$, $f(1)$  both even forces $2^2\parallel M_1$ or $2^4\mid M_1$,  contradicting $2^5\parallel M_1M_2$.

We can though get determinants of this form with $m=\frac{1}{2}(p^2+1)$:
\begin{align*}
f(x) & = 1+x^2, \; g(x)=(x-1)\Phi_p(x^2) \;\;\; \Rightarrow \;\;\; M=\frac{1}{2}(p^2+1)\:2^5,\\
f(x)&=-1+\mu h(x),\; g(x)=\Phi_p(-x)+\mu h(x) \;\;\; \Rightarrow \;\;\; M=- \frac{1}{2}(p^2+1)\:2^4 p^3 \mu,
\end{align*}
on observing that $1-\Phi_p(-\omega)\Phi_p(-\omega^{-1})=1-\frac{4}{(1+\omega)(1+\omega^{-1})}=\frac{(1-\omega)^2}{(1+\omega)^{2}}.$

When $p=1\mod 4$ we can write $2p=A^2+B^2$ and
$$ f(x)=(1+x)+A(x^p-1)\Phi_p(x^2), \; g(x)=B(x^p-1)\Phi_p(x^2)\;\;\; \Rightarrow \;\;\; M=2^5p^5. $$

For $p=5$ we also  get the missing values $2^5p^3$.
$$ f(x) =1-x+x^2+(1+x^p)x, \; g(x)=1+(1+x^p)(x+x^2)\;\;  \Rightarrow \;\; M=-2^5 p^3. \qedhere $$

\end{proof}

\end{document}